\documentclass{amsart}
\usepackage{amssymb,amsmath,latexsym,times,tikz,mathtools,enumitem,
hyperref,multirow,xcolor,accents}

\usetikzlibrary{patterns}

\numberwithin{equation}{section}

\hypersetup{colorlinks=true, citecolor=blue, linkcolor=blue}

\newtheorem{theorem}{Theorem}[section]
\newtheorem{definition}[theorem]{Definition}
\newtheorem{corollary}[theorem]{Corollary}
\newtheorem{proposition}[theorem]{Proposition}
\newtheorem{lemma}[theorem]{Lemma}

\newcommand\thickbar[1]{\accentset{\rule{.4em}{.8pt}}{#1}}

\newcommand{\del}{\backslash}
\newcommand{\ba}{\backslash}

\newcommand{\btu}{\bigtriangleup}

\title{The excluded 3-minors for vf-safe delta-matroids}
\author[J.~Bonin]{Joseph E.~Bonin} \address[J.~Bonin]
{Department of Mathematics\\ The George Washington University\\
  Washington, D.C. 20052, USA} \email[J.~Bonin] {jbonin@gwu.edu}
\author[C.~Chun] {Carolyn Chun} \address[C.~Chun]
{Department of Mathematics\\
  United States Naval Academy\\
  Annapolis, MD, 21402, USA} \email[C.~Chun] {chun@usna.edu}
\author[S.~Noble] {Steven D.\ Noble} \address[S.~Noble] {Department of
  Economics, Mathematics and Statistics \\ Birkbeck, University of
  London \\ London WC1E 7HX, United Kingdom} \email[S.~Noble]
{s.noble@bbk.ac.uk} \date{\today}

\begin{document}

\begin{abstract}
  Vf-safe delta-matroids have the desirable property of behaving well
  under certain duality operations.  Several important classes of
  delta-matroids are known to be vf-safe, including the class of
  ribbon-graphic delta-matroids, which is related to the class of
  ribbon graphs or embedded graphs in the same way that graphic
  matroids correspond to graphs.  In this paper, we characterize
  vf-safe delta-matroids and ribbon-graphic delta-matroids by finding
  the minimal obstructions, called excluded 3-minors, to membership in
  the class. We find the unique (up to twisted duality) excluded
  $3$-minor within the class of set systems for the class of vf-safe
  delta-matroids.  In the literature, binary delta-matroids appear in
  many different guises, with appropriate notions of minor operations
  equivalent to that of $3$-minors, perhaps most notably as graphs
  with vertex minors. We give a direct explanation of this equivalence
  and show that some well-known results may be expressed in terms of
  $3$-minors.
\end{abstract}

\maketitle

\section{Introduction}\label{sec:intro}

A \emph{set system} is a pair $S=(E,\mathcal{F})$, where $E$, or
$E(S)$, is a set, called the \emph{ground set}, and $\mathcal{F}$, or
$\mathcal{F}(S)$, is a collection of subsets of $E$.  (All set systems
in this paper have finite ground sets.)  The members of $\mathcal{F}$
are the \emph{feasible sets}.  We say that $S$ is \emph{proper} if
$\mathcal{F}\ne \emptyset$.

A matroid $M$ has many associated set systems with $E=E(M)$.  The most
important of these from the perspective of this paper has
$\mathcal{F}=\mathcal{B}(M)$, the set of bases of $M$.  Recall that
the bases of a matroid satisfy the following exchange property: for
any $B_1,B_2\in\mathcal{B}(M)$ and for each element $x\in B_1-B_2$,
there is a $y\in B_2-B_1$ for which
$B_1\triangle\{x,y\}\in\mathcal{B}(M)$.  To get the definition of a
delta-matroid, replace set differences by symmetric differences.
Thus, as introduced by Bouchet in \cite{ab1}, a \emph{delta-matroid}
is a proper set system $D=(E,\mathcal{F})$ for which $\mathcal{F}$
satisfies the \emph{delta-matroid symmetric exchange axiom}:
\begin{quote}
  (SE) \
  for all triples $(X,Y,u)$ with $X$ and $Y$ in $\mathcal{F}$
  and $u\in X\triangle Y $, there is a $v \in X\triangle Y$ (perhaps
  $u$ itself) such that $X\triangle \{u,v\}$ is in $\mathcal{F}$.
\end{quote}
Clearly every matroid $(E(M),\mathcal{B}(M))$ is a delta-matroid.

Just as there is a mutually-enriching interplay between matroid theory
and graph theory, the theory of delta-matroids has substantial
connections with the theory of embedded graphs or equivalently ribbon
graphs; see~\cite{CMNR,CCMR}.  Brijder and
Hoogeboom~\cite{B+H:group,B+H:interlace,B+H:nullity} introduced the
operation of loop complementation, which we define in the next
paragraph.  This operation is natural for the class of binary
delta-matroids and its subclass of ribbon-graphic delta-matroids.
These classes are closed under loop complementation, but that is not
true for the class of all delta-matroids.  We investigate when loop
complementation of a delta-matroid yields a delta-matroid.

For a set system $S=(E,\mathcal{F})$ and $e\in E$, we define $S+e$ to
be the set system
\begin{equation}\label{eq:loopcomp}
  S+e=(E,\mathcal{F}\btu\{F\cup e:e\notin F\in \mathcal{F}\}).
\end{equation}
(As in matroid theory, we often omit set braces from singletons.)
Note that $(S+e)+e=S$ and that $S+e$ is proper if and only if $S$ is
proper.  It is straightforward to check that if $e_1,e_2\in E$ then
$(S+e_1)+e_2=(S+e_2)+e_1$.  Thus if $X=\{e_1,\ldots,e_n\}$ is a subset
of $E$, then the set system $S+X$ is unambiguously defined by
\begin{equation}\label{eq:loopcomp2}
  S+X = ((S+e_1)+ \cdots)+e_n.
\end{equation}
This operation is called the \emph{loop complementation of $S$ on
  $X$}.  There is a natural operation of embedded graphs, namely
\emph{partial Petriality}, to which loop complementation
corresponds. More precisely if two embedded graphs are partial
Petrials of each other then their ribbon graphic delta-matroids are
related by a loop complementation~\cite[Section~4]{CCMR}.

For a delta-matroid $D$ and element $e\in E(D)$, the set system $D+e$
need not be a delta-matroid.  Consider, for example, the delta-matroid
$D_3=(\{a,b,c\},2^{\{a,b,c\}}-\{\{a,b,c\}\})$.  Then $D_3+\{a,b,c\}$
is the set system $(\{a,b,c\},\{\emptyset ,\{a,b,c\}\})$.  This is not
a delta-matroid.  In fact, it is an excluded minor for the class of
delta-matroids~\cite{BCN1}.

Another operation on delta-matroids is the twist.  For $A\subseteq E$,
the \emph{twist of $S$ on $A$}, which is also called the \emph{partial
  dual of $S$ with respect to $A$}, denoted $S*A$, is given by
\[S*A=(E,\{F\btu A\,:\,F\in\mathcal{F}\}).\]
Note that $(S*A)*A=S$.  The \emph{dual} $S^*$ of $S$ is $S*E$.  In
contrast with loop complementation, each twist of a delta-matroid is a
delta-matroid.  Apart from the dual, the twists of a matroid
$(E(M),\mathcal{B}(M))$ are generally not matroids, as discussed
in~\cite[Theorem~3.4]{CCMR}.

Two set systems are said to be \emph{twisted duals} of one another if
one may be obtained from the other by a sequence of twists and loop
complementations.  Following~\cite{B+H:nullity}, a delta-matroid is
said to be \emph{vf-safe} if all of its twisted duals are
delta-matroids.  (The term vf-safe is short for `vertex-flip
safe'. Both of the terms vf-safe and loop complementation are named
for operations on graphs representing binary
delta-matroids~\cite{B+H:group}, which we discuss in
Section~\ref{sec:vertexminor}.)

Delta-matroids belonging to certain important classes are known to be
vf-safe. In fact, every twisted dual of a ribbon-graphic delta-matroid
is a ribbon-graphic
delta-matroid~\cite[Theorem~2.1,Theorem~4.1]{CCMR}, and every twisted
dual of a binary delta-matroid is a binary
delta-matroid~\cite[Theorem~8.2]{B+H:nullity}.  Brijder and Hoogeboom
showed that quaternary matroids are vf-safe~\cite{BH}, although, as
mentioned earlier, matroids are not closed under twists.

In the main result of this paper, Theorem \ref{thm:excl3minvfs}, we
identify $D_3$ as essentially the unique obstacle for a delta-matroid
to be vf-safe.  More precisely, we show that the excluded
$3$-minors for the class of vf-safe delta-matroids within the class of
set systems comprise the 28 set systems that are the twisted duals of
$D_3$.  These set systems are given in Tables~\ref{t1}--\ref{t6}.  In
the final section of the paper, we relate $3$-minors to other minor
operations that have appeared in the literature, and we apply Theorem
\ref{thm:excl3minvfs} to recast some known results using short lists
of excluded $3$-minors.

\section{Background}

Let $S=(E,\mathcal{F})$ be a proper set system.  An element $e\in E$
is a \emph{loop} of $S$ if no set in $\mathcal{F}$ contains $e$.  If
$e$ is in every set in $\mathcal{F}$, then $e$ is a \emph{coloop}.  If
$e$ is not a loop, then the \emph{contraction of $e$ from $S$},
written $S/e$, is given by
\[S/e = (E-e, \{F-e:e\in F\in\mathcal{F}\}).\]
If $e$ is not a coloop, then the \emph{deletion of $e$ from $S$},
written $S\ba e$, is given by
\[S\ba e = (E-e,\{F\subseteq E-e:F\in\mathcal{F}\}).\]
If $e$ is a loop or a coloop, then one of $S/e$ and $S\ba e$ has
already been defined, so we can set $S/e=S\ba e$.  Any sequence of
deletions and contractions, starting from $S$, gives another set
system $S'$, called a \emph{minor} of $S$.  Each minor of $S$ is a
proper set system.

The order in which elements are deleted or contracted can matter.
See~\cite{BCN1} for an example.  However, for disjoint subsets $X$ and
$Y$ of $E$, if some set in $\mathcal{F}$ is disjoint from $X$ and
contains $Y$, then the deletions and contractions in $S\ba X/Y$ can be
done in any order, and
\[S\ba X/Y=(E-(X\cup Y),\{F-Y\,:\,F\in\mathcal{F} \text{ and
}Y\subseteq F\subseteq E-X\}).\]
The following lemma, which is \cite[Lemma~2.1]{BCN1}, shows that all
minors of a proper set system are of this type.

\begin{lemma}\label{minorisminor}
  For any minor $S'$ of a proper set system
  $S=(E,\mathcal{F})$, there are disjoint subsets $X$ and $Y$ of $E$
  such that
  \[S'=S\del X/Y=(E-(X\cup Y),\{F-Y\,:\,F\in\mathcal{F} \text{ and
  }Y\subseteq F\subseteq E-X\}).\]
\end{lemma}

Bouchet and Duchamp~\cite{bou:rep} showed that, if $S$ is a
delta-matroid, $X$ and $Y$ are any disjoint subsets of $E$, and
$S'=S\ba X/Y$, then $S'$ is a delta-matroid and $S'$ is independent of
the order of the deletions and contractions.

The following definition from~\cite{B+H:group} is equivalent to that
given in equations (\ref{eq:loopcomp})--(\ref{eq:loopcomp2}).
Equivalence can be shown by a routine induction on $|A|$.

\begin{definition}\label{def:loopcomp} If $S=(E,\mathcal F)$ is a set
  system and $A$ is a subset of $E$, then the \emph{loop
    complementation of $S$ by $A$}, denoted by $S+A$, is the set
  system on $E$ such that $F$ is feasible in $S+A$ if and only if $S$
  has an odd number of feasible sets $F'$ with
  $F-A \subseteq F' \subseteq F$.
\end{definition}

Note that if $A=\{e\}$, then the feasible sets of $S+e$ that do not
contain $e$ are the same as those of $S$, and a set $F$ containing $e$
is feasible in $S+e$ if and only if one but not both of $F$ and $F-e$
is feasible in $S$.  That is, so long as $e$ is not a loop or coloop,
$$\mathcal{F}(S+e)=\mathcal{F}(S\ba e)\cup \{F\cup
e:F\in\mathcal{F}(S\ba e)\btu\mathcal{F}(S/e)\}.$$
If $e$ is a loop, then
$\mathcal{F}(S+e)=\mathcal{F}\cup \{F\cup e:F\in\mathcal{F}\}$.  If
$e$ is a coloop, then $S+e=S$.

The twist and loop complementation operations interact in the
following way. If $A$ and $B$ are disjoint subsets of $E$ then
$(S+A)*B=(S*B)+A$ (a two-element case check and routine induction
suffice to verify this), but in general $(S*A)+A \ne (S+A)*A$. However
$((S+A)*A)+A = ((S*A)+A)*A$ (see~\cite{B+H:group}).  It follows that
there are at most six twisted duals of $S$ with respect to a fixed set
$A$.  These relations ensure that any twisted dual of $S$ is of the
form $((S*X)+Y)*Z$ for suitably chosen subsets $X$, $Y$ and $Z$ of $E$
with $X \subseteq Z$.  The first relation is used in the proof of the
following result.

\begin{lemma}\label{reordering}
  Let $S=(E,\mathcal F)$ be a proper set system, and let $a$ and $b$
  be distinct elements of $E$.  Then
  \begin{enumerate}[label=\emph{(\roman*)}]
  \item $S+a\ba a=S\ba a$,
  \item $S+a\ba b = S\ba b+a$, and
  \item $S+a/b = S/b+a$.
  \end{enumerate}
\end{lemma}

\begin{proof}
  If $a$ is a coloop of $S$, then $S+a=S$, so statement (i) follows.
  Also, $a$ is not a coloop of $S$ if and only if it is not a coloop
  of $S+a$, in which case the feasible sets avoiding $a$ are the same
  in $S$ and $S+a$ by the definition.

  For statement (ii), observe that $b$ is a coloop of $S+a$ if and
  only if it is a coloop of $S$.  When $b$ is not a coloop of $S$,
  statement (ii) holds since for each side, the feasible sets are the
  sets $F$ with $b\not\in F$ for which an odd number of the sets $X$
  with $F-a\subseteq X\subseteq F$ are in $\mathcal F$.  When $b$ is a
  coloop of $S$, we need to show that $S+a/b = S/b+a$.  This holds
  since for each side, the feasible sets are the sets $F$ with
  $b\not\in F$ for which an odd number of the sets $X$ with
  $(F-a)\cup b\subseteq X\subseteq F\cup b$ are in $\mathcal F$.

  It is easy to check that $S'/e=S'*e\del e$, so, using statement
  (ii), we get statement (iii):
  \[S+a/b=((S+a)*b)\del b=((S*b)+a)\del b=((S*b)\del
  b)+a=S/b+a.\qedhere\]
\end{proof}

The counterpart, for contractions, of statement (i) is false, as
taking $S=D_3$ shows.

\section{$3$-minors}\label{3minorsection}
We introduce a third minor operation on set systems.  For a proper set
system $S$, we define $S\ddagger e$ to be the set system
$(S+e)/e$. This minor operation was first defined by Ellis--Monaghan
and Moffatt~\cite{E+M:Penrose} for ribbon graphs and in an equivalent
way by Brijder and Hoogeboom~\cite{B+H:interlace} for
delta-matroids. The notation $\ddagger$ is new, but it seems
appropriate to shorten the unwieldy $+e/e$ notation. Motivation for
this definition comes from two directions. First, one way to define
the Penrose polynomial of a ribbon graph is by specifying a recursive
relation analogous to the deletion-contraction recurrence of the
chromatic polynomial with this minor operation replacing
contraction. For this reason, following a suggestion of Iain
Moffatt~\cite{Moffatt:private}, we propose calling the operation
\emph{Penrose contraction}. Second, there is a class of combinatorial
objects called multimatroids~\cite{mmi,mmii,mmiii}, of which tight
3-matroids are a particular subclass.  Brijder and
Hoogeboom~\cite{B+H:interlace} showed that tight 3-matroids are
equivalent (in a sense that we do not make precise here) to vf-safe
delta-matroids. Tight 3-matroids have three minor operations
corresponding to deletion, contraction, and Penrose contraction in
vf-safe delta-matroids.

Any sequence of the three minor operations, starting from $S$, gives
another proper set system $S'$, called a \emph{$3$-minor} of $S$.  A
collection $\mathcal C$ of proper set systems is \emph{$3$-minor
  closed} if every $3$-minor of every member of $\mathcal C$ is in
$\mathcal C$.  Given such a collection $\mathcal C$, a proper set
system $S$ is an \emph{excluded $3$-minor} for $\mathcal{C}$ if
$S\notin\mathcal{C}$ and all other $3$-minors of $S$ are in
$\mathcal{C}$.  A proper set system belongs to $\mathcal C$ if and
only if none of its $3$-minors is an excluded $3$-minor for
$\mathcal C$.  Thus, the excluded $3$-minors determine $\mathcal C$;
they are the $3$-minor-minimal obstructions to membership in
$\mathcal{C}$.

For a given class $\mathcal{C}$, there may be substantially fewer
excluded $3$-minors than excluded minors.  For instance,
in~\cite{oumg}, Geelen and Oum found 166 delta-matroids that, up to
twists, are the excluded minors for ribbon-graphic delta-matroids
within the class of binary delta-matroids.  In contrast, in
Theorem~\ref{thm:rg3min}, we show that every binary matroid that does
not have a twisted dual of one of three delta-matroids as a $3$-minor
is ribbon-graphic.

An element $e$ is called a \emph{pseudo-loop} of $S$ if $e$ is a loop
of $S + e$. The next lemma characterizes these elements.

\begin{lemma}\label{collection}
  For an element $e$ in a proper set system $S$, the following
  statements are equivalent:
  \begin{enumerate}
  \item[\emph{(i)}] $e$ is a loop of $S+e$, that is, a pseudo-loop of
    $S$,
  \item[\emph{(ii)}] $F\cup e\in\mathcal{F}(S)$ if and only if
    $F\in\mathcal{F}(S)$, and
  \item[\emph{(iii)}] $S*e=S$.
  \end{enumerate}
  Pseudo-loops of $S$ are neither loops nor coloops of $S$.
  Furthermore, $S\ddagger e=S \setminus e = S/e$ if and only if $e$ is
  a loop, a coloop, or a pseudo-loop of $S$.
\end{lemma}

\begin{proof}
  The equivalence of statements (i)--(iii) is immediate from the
  definitions. Statement (ii) implies that pseudo-loops are neither
  loops nor coloops.  If $e$ is a loop of $S$, then
  $S\ddagger e=S \setminus e$ since
  $\mathcal{F}(S+e)=\mathcal{F}(S)\cup \{F\cup
  e:F\in\mathcal{F}(S)\}$;
  also, $S \setminus e = S/e$ by definition.  If $e$ is a coloop of
  $S$, then $S\ddagger e=S / e$ since $S+e=S$; also,
  $S \setminus e = S/e$ by definition.  If $e$ is a pseudo-loop of
  $S$, then statements (i) and (ii) gives the equality.  If $e$ is not
  a loop, a coloop, or a pseudo-loop of $S$, then
  $S \setminus e \ne S/e$ by the failure of statement (ii) and the
  fact that some, but not all, sets in $\mathcal{F}(S)$ contain $e$.
\end{proof}

The following two results show that one may choose the operations used
to form a $3$-minor in such a way that they commute.

\begin{lemma}\label{3point2}
  Let $S=(E,\mathcal F)$ be a proper set system, and let $X,Y$, and
  $Z$ be pairwise disjoint subsets of $E$.  If there is a set $F$ with
  \begin{equation}\label{eq:3minornice2} F \subseteq E-(X\cup Y \cup
    Z)
    \quad \text{and} \quad |\mathcal F \cap \{F': F\cup Y \subseteq F'
    \subseteq F \cup Y\cup Z\}| \text{ is odd},
  \end{equation}
  then the minor operations in $S\setminus X/Y\ddagger Z$ can be done
  in any order and a set $F$ is feasible in $S\setminus X/Y\ddagger Z$
  if and only if it satisfies Condition~\eqref{eq:3minornice2}.
\end{lemma}

\begin{proof}
  A set $F$ meets Condition~\eqref{eq:3minornice2} if and only if
  $F\subseteq E-(X\cup Y \cup Z)$ and $F\cup Y\cup Z$ is in
  $\mathcal{F}(S+Z)$. If there is at least one set satisfying
  Condition~\eqref{eq:3minornice2}, the remarks preceding
  Lemma~\ref{minorisminor} imply that the deletions and contractions
  in forming $(S+Z)\setminus X/(Y\cup Z)$ from $S+Z$ may be done in
  any order and a set $F$ is feasible in $(S+Z)\setminus X/(Y\cup Z)$
  if and only if it satisfies
  Condition~\eqref{eq:3minornice2}. Lemma~\ref{reordering} implies
  that we may defer taking a loop complementation of an element in $Z$
  until just before it is contracted. The result follows.
\end{proof}

We next show that for every $3$-minor of a proper set system, there
are pairwise disjoint sets $X$, $Y$ and $Z$ satisfying
Condition~\eqref{eq:3minornice2}.

\begin{lemma}\label{xyzlemma}
  Let $S'$ be a $3$-minor of a proper set system $S=(E,\mathcal F)$.
  Then there are pairwise disjoint subsets $X$, $Y$, and $Z$ of $E$
  such that $S'=S\setminus X/Y\ddagger Z$ and there is a set $F$
  satisfying Condition~\eqref{eq:3minornice2}.
\end{lemma}

\begin{proof}
  Suppose we get $S'$ from $S$ by, for each of $e_1,e_2,\ldots,e_k$ in
  turn, performing one the three minor operations, giving the sequence
  of minors $S_0=S,S_1,\ldots,S_k=S'$.  Let $X$ be the set of elements
  $e_i$ in $\{e_1,\ldots,e_k\}$ that satisfy at least one of the
  following conditions:
  \begin{enumerate}
  \item $e_i$ is a loop or a pseudo-loop of $S_{i-1}$, so
    $S_i=S_{i-1}\setminus e_i$, or
  \item $e_i$ is not a coloop of $S_{i-1}$ and
    $S_i=S_{i-1}\setminus e_i$.
  \end{enumerate}
  Let $Y$ be the set of elements $e_i$ in $\{e_1,\ldots,e_k\}-X$ such
  that $e_i$ is either a coloop of $S_{i-1}$ or
  $S_i=S_{i-1}/e_i$. Note that if $e_i\in Y$ then it is not a loop in
  $S_{i-1}$. Finally let $Z=\{e_1,\ldots,e_k\}-(X\cup Y)$, so that $Z$
  comprises the elements $e_i$ in $\{e_1,\ldots,e_k\}$ for which
  $S_i=S_{i-1}\ddagger e_i$ but $e_i$ is not a loop, pseudo-loop or
  coloop. Then there is always at least one set $F$ satisfying
  Condition~\eqref{eq:3minornice2}.
\end{proof}

It follows from the definition of minors in a multimatroid~\cite{mmii}
and the equivalence of tight $3$-matroids and vf-safe
delta-matroids~\cite{B+H:interlace}, that if $S$ is a vf-safe
delta-matroid, then changing the order of a sequence of $3$-minor
operations on $S$ never affects the result.

Table~\ref{tab:minors} shows the result of applying one of the three
minor operations that remove $e$ after taking one of the six twisted
duals, with respect to $e$, of a proper set system.  If instead the
minor operation removes a different element from that used for the
twisted dual, then these operations commute.

\begin{table}
\begin{tabular}{c|ccc} & $/e$ & $\setminus e$ & $\ddagger e$ \\
\hline
$S$\rule{0pt}{11pt}  &   $S/e$  &$S\setminus e$  & $S\ddagger e$\\
$S*e$&   $S\setminus e$     &$S/e$  &$S\ddagger e$ \\
$S+e$& $S\ddagger e$       &$S\setminus e$  &$S/e$ \\
$(S+e)*e$&$S\setminus e$       & $S\ddagger e$ &$S/e$ \\
$(S*e)+e$& $S\ddagger e$      &$S/e$  &$S\setminus e$ \\
$((S*e)+e)*e$&$S/e$       &$S\ddagger e$  &$S\setminus e$
\end{tabular}
\caption{Interaction of minor operations and twisted
  duality. \protect\rule{0pt}{11pt}\label{tab:minors}}
\end{table}

We next show that any $3$-minor of a twisted dual of a proper set
system $S$ is a twisted dual of some $3$-minor of $S$.  It is easy to
see that the converse is also true.

\begin{lemma}\label{lem:3minorsdualsstuff}
  Suppose $S$ is a proper set system and $S'$ is a twisted dual of
  $S$.  If $S''$ is a $3$-minor of $S'$, then $S$ has a $3$-minor that
  is a twisted dual of $S''$.
\end{lemma}

\begin{proof}
  There are subsets $A$ and $B$ of $E(S)$ such that we obtain $S''$
  from $S$ by first forming a twisted dual for each element of $A$ and
  then performing one of the three minor operations for each element
  of $B$. The remarks before this lemma imply that one may reorder
  these operations to first deal with the elements of $A\cap B$, one
  by one, forming a twisted dual for an element and then a $3$-minor
  before moving on to the next element. According to
  Table~\ref{tab:minors} each of these pairs of operations may be
  replaced by a single $3$-minor operation.  Next a $3$-minor is
  formed for each element of $B- A$. The resulting set system is a
  twisted dual of $S''$ with respect to the elements of $A-B$.
\end{proof}

\section{Characterizations by excluded $3$-minors}

Brijder and Hoogeboom \cite{B+H:nullity} showed that the class of
vf-safe delta-matroids is minor-closed.  A computer search for
excluded minors for this class turns up many examples with apparently
little structure.  The class of vf-safe delta-matroids is defined
using both the twist and loop complementation operations, so it is
natural to try to characterize this class using $3$-minors.  By Lemma
\ref{lem:pcclosed} below, the class of vf-safe delta-matroids is
closed under Penrose contraction, so, with the result in
\cite{B+H:nullity}, it is closed under $3$-minors.  The main result of
this section, Theorem \ref{thm:excl3minvfs}, gives the excluded
$3$-minors for the class of vf-safe delta-matroids within the class of
set systems.

\begin{lemma}\label{lem:pcclosed}
  If $S$ is vf-safe and $e\in E(S)$, then $S\ddagger e$ is vf-safe.
\end{lemma}

\begin{proof}
  If $S$ is vf-safe, then all of its twisted duals are vf-safe by
  definition, so $S+e$ is vf-safe.  Theorem~8.3 in~\cite{B+H:nullity}
  states that every minor of a vf-safe delta-matroid is vf-safe. Thus
  $S\ddag e=S+e/e$ is vf-safe.
\end{proof}

Let \[S_i=(\{e_1,e_2,\dots ,e_i\},\{\emptyset ,\{e_1,e_2,\dots ,e_i\}\}).\]
Let $\mathcal{S}$ be the set of all twists of the set systems in
$\{S_3,S_4,\dots \}$.  Let
\begin{itemize}
\item $T_1=(\{a,b,c\},\{\emptyset ,\{a,b\},\{a,b,c\}\})$;
\item $T_2=(\{a,b,c\},\{\emptyset ,\{a,b\},\{a,c\},\{a,b,c\}\})$;
\item $T_3=(\{a,b,c\},\{\emptyset ,\{a\},\{a,b\},\{a,b,c\}\})$;
\item $T_4=(\{a,b,c\},\{\emptyset,\{a\},\{a,b\},\{a,c\},\{a,b,c\}\})$;
\item $T_5=(\{a,b,c,d\},\{\emptyset ,\{a,b\},\{a,b,c,d\}\})$;
\item $T_6=(\{a,b,c,d\},\{\emptyset ,\{a,b\},\{a,c\},\{a,b,c,d\}\})$;
\item $T_7=(\{a,b,c,d\},\{\emptyset,\{a,b\},\{a,c\},\{a,d\},
  \{a,b,c,d\}\})$;
\item $T_8=(\{a,b,c,d\},\{\emptyset,\{a\},\{a,b\},\{a,c\},
  \{a,d\},\{a,b,c,d\}\})$.
\end{itemize}
Let $\mathcal{T}$ be the set of all twists of the set systems in
$\{T_1,T_2, \dots ,T_8\}$.  By the following result
from~\cite[Theorem~5.1]{BCN1}, these are all of the excluded minors
for delta-matroids within the class of set systems.

\begin{theorem}
  A proper set system $S$ is a delta-matroid if and only if $S$ has no
  minor isomorphic to a set system in $\mathcal{S}\cup\mathcal{T}$.
\end{theorem}

The following lemma is key for finding the excluded $3$-minors for
vf-safe delta-matroids within the class of set systems.

\begin{lemma}\label{lem:3minorvf}
  Let $S$ be an excluded $3$-minor for the class of vf-safe
  delta-matroids. Then $S$ has a twisted dual that is isomorphic to a
  set system in $\mathcal S \cup \mathcal T$.
\end{lemma}

\begin{proof}
  Such an excluded $3$-minor $S$ either is not a delta-matroid and all
  of its other minors are delta-matroids, or it is a delta-matroid and
  has a twisted dual $S'$ that is not a delta-matroid. In the former
  case $S$ is isomorphic to a set system in
  $\mathcal S \cup \mathcal T$ and the lemma holds. In the latter case
  $S'$ has a minor $S''$ isomorphic to a member of
  $\mathcal S \cup \mathcal T$. By Lemma~\ref{lem:3minorsdualsstuff},
  $S$ has a $3$-minor $S'''$ that is a twisted dual of
  $S''$. Therefore $S'''$ is not a vf-safe delta-matroid. The only
  $3$-minor of $S$ that is not a vf-safe delta-matroid is $S$
  itself. Hence $S=S'''$ and the lemma holds.
\end{proof}

To connect the next result with the remarks in Section
\ref{sec:intro}, note that $D_3+\{a,b,c\} =S_3$.

\begin{theorem}\label{thm:excl3minvfs}
  A proper set system is a vf-safe delta-matroid if and only if it has
  no $3$-minor that is isomorphic to a twisted dual of $S_3$.
\end{theorem}

\begin{proof}
  All proper set systems with two elements are delta-matroids, and
  therefore each one is vf-safe, so the twisted duals of $S_3$ are
  excluded $3$-minors for the class of vf-safe delta-matroids.  By
  Lemma~\ref{lem:3minorvf} every excluded $3$-minor for the class of
  vf-safe delta-matroids must be a twisted dual of a set system in
  $\mathcal S \cup \mathcal T$. We first consider the set systems with
  three-element ground sets.  We have $T_1^*+c=S_3$ and
  $T_2^*+\{b,c\}\simeq T_3+a=T_1$ and $T_4+a=T_2$, so every excluded
  $3$-minor of size three is a twisted dual of $S_3$.

  Lastly, we show that no other set system in
  $\mathcal S \cup \mathcal T$ is an excluded $3$-minor.
  Lemma~\ref{lem:3minorsdualsstuff} implies that each twisted dual of
  an excluded $3$-minor is an excluded $3$-minor, so it suffices to
  show that each of $T_5$, $T_6$, $T_7$, $T_8$, and $S_n$, for
  $n\geq 4$, has a smaller member of $\mathcal{S}\cup\mathcal{T}$ as
  a $3$-minor.  Indeed, $S_n \ddagger e_n = S_{n-1}$, for $n \geq 4$,
  $T_5\ddag d=T_1$, $T_6\ddag d= T_8\ddag d=T_2$, and
  $T_7\ddag d=T_4$.
\end{proof}

As stated in the introduction, there are 28 twisted duals of $S_3$, up
to isomorphism.  These excluded $3$-minors are listed in
Tables~\ref{t1}--\ref{t6}.

\section{$3$-minors and vertex minors}\label{sec:vertexminor}

In this section we explain the link between $3$-minors of binary
delta-matroids and vertex minors of graphs.  As we shall see later, up
to twisted duality, each binary delta-matroid may be represented by a
loopless simple graph, and this equivalence is preserved under
appropriate $3$-minor operations in the binary delta-matroid and
vertex minors in the graph.

Vertex minors are well-studied, but are defined only for graphs. In
contrast, the concept of a $3$-minor is relatively unstudied, but is
important due to its direct correlation with ribbon-graph operations
and its applicability beyond binary delta-matroids. By combining
several existing results from the literature, one may show that the
notions of $3$-minors and vertex minors are equivalent in a sense to
be made precise later, but to see this equivalence one must piece
together results involving several combinatorial structures, which are
not obviously related. For this reason, and for completeness, we give
a full explanation of the link here. Although the key ideas presented
here are not new, as far as we know this direct link has not
previously been fully described.

A delta-matroid is \emph{normal} if the empty set is feasible. A
delta-matroid is \emph{even} if for every pair $F_1$ and $F_2$ of its
feasible sets $|F_1 \btu F_2|$ is even. Equivalently, the sizes of all
its feasible sets are of the same parity.  Let $M$ denote a symmetric
binary matrix with rows and columns indexed by $[n]=\{1,\ldots,n\}$.
Take $E=[n]$ and $\mathcal F$ to be the collection of subsets $S$ of
$[n]$ for which the principal submatrix of $M$ comprising the rows and
columns indexed by elements of $S$ is
non-singular. Bouchet~\cite{bou:rep} showed that $D(M)=(E,\mathcal F)$
is a delta-matroid. We call such delta-matroids \emph{basic binary}.
(In the literature, what we have called basic binary delta-matroids
are often called graphic, but we prefer to avoid this term to prevent
confusion with ribbon-graphic delta-matroids.)  A delta-matroid is
\emph{binary}~\cite{bou:rep} if it is a twist of a basic binary
delta-matroid.

It follows immediately from the definition that every basic binary
delta-matroid is normal and that a basic binary delta-matroid is
uniquely determined by its feasible sets of size at most two.  A
well-known result of linear algebra states that a symmetric matrix
with an odd number of rows (and columns) and zero diagonal is
singular.  Consequently a basic binary delta-matroid is even if and
only if it has no feasible sets of size one.

Let $A$ be a matrix over an arbitrary field with rows and columns
indexed by $[n]$, and let $X$ be a subset of $[n]$ such that the
principal submatrix $P=A[X]$ is non-singular. Suppose without loss of
generality that $A=\begin{pmatrix}P&Q\\R&S\end{pmatrix}$. Then the
matrix $A*X$ is defined by
\[ A*X
= \begin{pmatrix}P^{-1}&-P^{-1}Q\\RP^{-1}&S-RP^{-1}Q\end{pmatrix}.\]
Note that if $A$ is a symmetric binary matrix then $A*X$ is
symmetric. The following result is due to Tucker~\cite{Tucker}.

\begin{theorem}
  Let $A$ be a matrix over an arbitrary field with rows and columns
  indexed by $[n]$, and let $X$ be a subset of $[n]$ such that the
  principal submatrix $P=A[X]$ is non-singular. Then for every subset
  $Y$ of $[n]$, we have
\[\det\bigl((A*X)[Y]\bigr) = \frac{\det(A[X\bigtriangleup Y])}{\det(A[X])}.\]
\end{theorem}

In particular for any subset $Y$ of $[n]$, the principal submatrix
$(A*X)[Y]$ is non-singular if and only if the principal submatrix
$A[X\bigtriangleup Y]$ is non-singular.

The following corollary is immediate.

\begin{corollary}\label{cor:twistrep}
  Suppose that $A$ is a binary matrix, and $X$ is a feasible set of
  $D(A)$. Then $D(A)*X=D(A*X)$.
\end{corollary}

See~\cite{bou:rep} for an alternative proof of this result that holds
for arbitrary fields. A consequence of this corollary is that every
normal twist of a basic binary delta-matroid is basic binary.

A \emph{looped simple graph} is a graph without parallel edges but in
which each vertex is allowed to have up to one loop.  Much of the time
we will forbid loops; we call such graphs \emph{loopless simple
  graphs}.  Recall that basic binary delta-matroids are completely
determined by their feasible sets with size two or fewer.  Clearly
basic binary delta-matroids on the set $[n]$ are in one-to-one
correspondence with looped simple graphs with vertex set $[n]$;
likewise, even basic binary delta-matroids on $[n]$ are in one-to-one
correspondence with loopless simple graphs with vertex set $[n]$.

We take adjacency matrices to always be binary.  Given a looped simple
graph $G$ and its adjacency matrix $A$, we let $D(G)$ denote the basic
binary delta-matroid $D(A)$.  If $X$ is a subset of the edges of $G$,
then $X$ labels a subset of the rows and columns of $A$, and we define
$G*X$ to be the looped simple graph with adjacency matrix $A*X$.

We now consider various transformations that may be applied to $G$
and their effect on $D(G)$.

The loop complementation operation of Brijder and Hoogeboom was first
defined in terms of basic binary delta-matroids. If $G$ is a looped
simple graph and $v$ is a vertex of $G$, then the loop complementation
$G+v$ is formed by toggling the loop at $v$, that is, removing a loop
if there is one at $v$ and adding one at $v$ if there is no loop
there.

The following lemma from~\cite{B+H:group} is straightforward.

\begin{lemma}
  Let $G$ be a looped simple graph with vertex $v$. Then
  $D(G+v)=D(G)+v$.
\end{lemma}

Our next operation is local complementation. There are several
variations in the definition of local complementation: see, for
instance,~\cite{Traldi:binloc}. We will only require certain cases of
what is defined there.  For a looped simple graph $G$ with vertex $v$,
let $N_G(v)$ denote the \emph{open neighbourhood} of $v$, that is, the
set of vertices, excluding $v$, that are adjacent to $v$ in $G$. We
will generally write $N$ instead of $N_G$ when there is no possibility
of confusion.  The \emph{local complementation} of $G$ at $v$, denoted
by $G^v$, is formed by toggling the adjacencies between pairs of
neighbours of $v$, that is, for every distinct pair $x$, $y$ from the
neighbourhood of $v$, delete edge $xy$ if $x$ and $y$ are adjacent in
$G$ and add edge $xy$ if $x$ and $y$ are not adjacent in
$G$. Additionally, if there is a loop at $v$, then the loop status of
every vertex in the open neighbourhood of $v$ is toggled.  In both
cases, adjacencies involving one or more non-neighbours of $v$ or $v$
itself are unchanged and the presence or not of a loop at $v$ is
unaffected. To distinguish the two cases, it will be helpful to refer
to local complementation at $v$ as \emph{simple local complementation}
when $v$ is loopless, and \emph{non-simple local complementation} when
there is a loop at $v$.

For delta-matroid $D$ and subset $A\subseteq E(D)$, let $D\thickbar{*}A$
denote the
\emph{dual pivot on $A$}, which is equal to $D+A*A+A$.  The following
result is implicit in the results of~\cite{Traldi:binloc}, but is not
expressed in this form.

\begin{proposition}\label{prop:localcomp}
  Let $G$ be a loopless simple graph with vertex $v$. Then
  $D(G^v)=(D(G)\thickbar{*} v) + N(v)$.
\end{proposition}

\begin{proof}
  Let $A$ be the adjacency matrix of $G$. Then $A$ is symmetric and
  all of its diagonal entries are zero.  Notice that the simple local
  complementation $G^v$ can be formed by adding a loop at $v$,
  performing the non-simple local complementation at $v$ and then
  removing the loops added at $v$ and all of its neighbours.

  We have $D(G+v)=D(G)+v$.  Assume without loss of generality that
  $v=1$ and let $Z=[n]-1$. Then the adjacency matrix of $G+v$ is
  $\begin{pmatrix}1& c \\ c^t & A[Z]\end{pmatrix}$ for some vector
  $c$. Then it follows from Corollary~\ref{cor:twistrep} that
  $(D(G)+v)*v = D((G+v)*v)=D(A')$ where
  $A'=\begin{pmatrix}1& c \\ c^t & A[Z]+c^tc\end{pmatrix}$.

  A diagonal entry of $c^tc$ is non-zero if it corresponds to a
  neighbour of $v$ and an off-diagonal entry of $c^tc$ is non-zero if
  both the row and column correspond to neighbours of $v$. Thus
  $(D(G)+v)*v = D(G')$ where $G'$ is formed from $G$ by adding a loop
  at $v$ and performing the non-simple local complementation at
  $v$. Now $G'$ has loops at $v$ and at all neighbours of $v$, so
  \begin{equation*}
    D(G^v)=D(G'+v+N(v))=D(G')+v+N(v)=(D(G)\thickbar{*} v) + N(v).\qedhere
  \end{equation*}
\end{proof}

The corollary below is well-known and follows from the previous
result.

\begin{corollary}
  Let $G$ be a loopless simple graph with adjacent vertices $v$ and
  $w$. Then $D(((G^v)^w)^v)=D(G)*\{v,w\}$.
\end{corollary}

\begin{proof}
  We have
  \[ D(((G^v)^w)^v) = ((D(G)\thickbar{*} v +N(v)) \thickbar{*} w +
  N_{G^v}(w))\thickbar{*} v + N_{(G^v)^w}(v).\]
  It follows from the discussion before Lemma~\ref{reordering} that
  one may reorder the loop complement and twist operations so that
  those involving a particular vertex of $G$ are done
  consecutively. The result follows by considering the effect of the
  operations involving each vertex of $G$ separately and noting that
  \begin{enumerate}
  \item a common neighbour of $v$ and $w$ in $G$ is a neighbour of $v$
    but not $w$ in both $G^v$ and $(G^v)^w$,
  \item a vertex other than $w$ that is a neighbour of $v$ but not $w$
    in $G$ is a neighbour of both $v$ and $w$ in $G^v$ and of $w$ but
    not $v$ in $(G^v)^w$, and
  \item a vertex other than $v$ that is a neighbour of $w$ but not $v$
    in $G$ is a neighbour of both $v$ and $w$ in $(G^v)^w$ and of $w$
    but not $v$ in $G^v$.\qedhere
  \end{enumerate}
\end{proof}

A \emph{vertex minor} of a looped simple graph $G$ is formed from $G$
by a sequence of local complementations and deletions of vertices. It
is easy to check that if $v$ and $w$ are different vertices of an
unlooped simple graph, then $(G^v)\setminus w = (G\setminus w)^v$ and
thus one may assume that all the local complementations are done
first.

\begin{figure}
\includegraphics[scale=1]{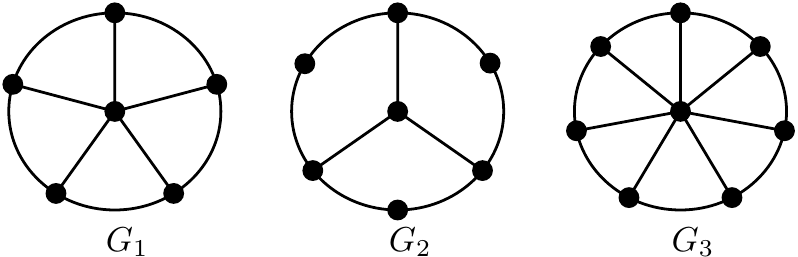}
\caption{A complete set of circle graph obstructions.}
\label{circleobst}
\end{figure}

Perhaps the most important use of vertex minors is Bouchet's
characterization of circle graphs.  A \emph{chord diagram} is a
collection of chords of a circle. Chord diagrams are in one-to-one
correspondence with orientable ribbon graphs with one vertex.  (For
more information on ribbon graphs, see~\cite{EMMbook} or~\cite{CCMR}.)
To see this think of the circle and its interior as the vertex of a
ribbon graph and for each chord attach a ribbon to the vertex at the
points corresponding to the endpoints of the chord.  Clearly two
chords intersect if and only if the corresponding ribbons $e_1$ and
$e_2$ are interlaced, that is, as one traverses the vertex one meets
an end of $e_1$, then an end of $e_2$, then the other end of $e_1$,
and finally the other end of $e_2$.  An unlooped simple graph is a
\emph{circle graph} if it is the intersection graph of the chords in a
chord diagram, that is, there is a vertex corresponding to each chord
and they are adjacent if and only if the chords cross. Equivalently a
circle graph is the interlacement graph of an orientable ribbon graph
with one vertex: it has a vertex for each ribbon and two vertices are
adjacent if the corresponding ribbons are interlaced. Bouchet
established the following result~\cite{Bouchet:circle}.

\begin{theorem}\label{Bou:famouscircle}
  An unlooped simple graph is a circle graph if and only if it has no
  vertex minor isomorphic to the graphs $G_1$, $G_2$ or $G_3$ depicted
  in Figure~\ref{circleobst}.
\end{theorem}

We are now ready to state the link between $3$-minors and vertex minors.

\begin{theorem}
\begin{enumerate}[leftmargin=*]
\item[\emph{(1)}] Let $G$ be a unlooped simple graph and let $H$ be a
  vertex minor of $G$. Then $D(H)$ is a $3$-minor of $D(G)$.
\item[\emph{(2)}]  Let $D$ be a twisted dual of a basic binary delta-matroid and
  let $D'$ be a $3$-minor of $D$. Then there are graphs $G$ and $G'$
  such that $D(G)$ and $D(G')$ are twisted duals of $D$ and $D'$
  respectively, and $G'$ is a vertex minor of $G$.
\end{enumerate}
\end{theorem}

\begin{proof}
  For part (1), note that a vertex minor of an unlooped simple graph
  is obtained by a sequence of local complementations and vertex
  deletions. The result follows from Proposition~\ref{prop:localcomp}
  and the fact that if $v$ is a vertex of $G$ then
  $D(G\setminus v)=D(G)\setminus v$.

  For part (2), let $F$ be a feasible set of $D$ and let
  \[ B= \{e\in E(D)\,:\, \{e\} \in \mathcal {F}(D*F)\}.\]
  The remark following Corollary~\ref{cor:twistrep} implies that $D*F$
  is basic binary, so $(D*F)+B$ is an even basic binary delta-matroid,
  so $(D*F)+B=D(G)$ for some unlooped simple graph $G$. It follows
  from Lemma~\ref{lem:3minorsdualsstuff} that there is a delta-matroid
  $D''$ that is a $3$-minor of $D(G)$ and a twisted dual of $D'$.  We
  shall prove by induction on $k$ that if $G$ is an unlooped simple
  graph and $D''$ is a $3$-minor of $D(G)$ with $k$ fewer elements,
  then there is an unlooped simple graph $G'$ that is a vertex minor
  of $G$ and such that $D(G')$ is a twisted dual of $D''$. The result
  then follows.

  If $k=0$, then take $G'=G$. Otherwise $D''$ is obtained from $D(G)$
  by a sequence of $k$ deletions, contractions and Penrose
  contractions. Suppose that the first operation is the deletion of
  $v$. Then take $G''=G\setminus v$, which is a vertex minor of
  $G$. Furthermore $D(G)\setminus v=D(G'')$ and $D''$ is a $3$-minor
  of $D(G'')$ with $k-1$ fewer edges. Hence, by induction, there is an
  unlooped simple graph $G'$ that is a vertex minor of $G''$ and hence
  of $G$, and such that $D(G')$ is a twisted dual of $D''$. Suppose
  next that the first operation is the Penrose contraction of
  $v$. Then take $G''=(G^v)\setminus v$. We have
  \begin{align*}
    D(G'')&=D(G^v\setminus v) \\
          & = ((((D(G)+v)*v)+v)+N(v))\setminus v\\
          &= ((((D(G)*v)+v)*v)\setminus v)+N(v) \\
          &= (((D(G)*v)+v)/ v) + N(v)\\
          &= (D(G) \ddag v) + N(v).
  \end{align*}
  (The last equality uses Table~\ref{tab:minors}.)  Now $D(G'')$ is a
  twisted dual of $D(G)\ddag v$, so it has a $3$-minor $D'''$ with
  $k-1$ fewer elements that is a twisted dual of $D''$. Hence, by
  induction, there is an unlooped simple graph $G'$ that is a vertex
  minor of $G''$ such that $D(G')$ is a twisted dual of $D'''$ and
  consequently of $D''$. In the final case the first operation is the
  contraction of $v$. If $v$ is an isolated vertex of $G$ then $v$
  appears in no feasible set of $D(G)$ of size at most two and
  consequently in no feasible set of $D(G)$ of any size. Thus $v$ is a
  loop of $D(G)$ and $D(G)/v=D(G)\setminus v=D(G\setminus v)$. If $v$
  is not an isolated vertex of $v$ then let $w$ be a neighbour of
  $v$. We have
  \begin{align*}
    D(((G^v)^w)^v\setminus v) &= D(((G^v)^w)^v)\setminus v\\
                              &= (D(G)*\{v,w\})\setminus v\\
                              & = (D(G)/v)*w.
  \end{align*} The
  proof of this case is completed in a similar way to the case of
  Penrose contraction.
\end{proof}

Having described the equivalence of $3$-minors and vertex minors, we
note how the equivalence may be seen from existing results involving
binary delta-matroids and other combinatorial structures, namely,
isotropic systems, binary tight 3-matroids (both defined by Bouchet
in~\cite{ab:iso1} and~\cite{mmi}, respectively) and isotropic matroids
(defined by Traldi in~\cite{Traldi:binloc}). The equivalence of
$3$-minors of binary delta-matroids and minors of binary tight
$3$-matroids is explained in \cite{B+H:interlace}. The equivalence of
minors of binary tight $3$-matroids, isotropic matroids and isotropic
systems is explained in~\cite{BT}. Finally vertex minors of graphs are
equivalent to minors of isotropic matroids by~\cite{Traldi:binloc} or
to minors of isotropic systems by~\cite{ab:iso2}.

From the preceding result we obtain the following translation of
Bouchet's result, determining the three binary delta-matroids that are
the excluded $3$-minors for ribbon-graphic delta-matroids. In
\cite{Bouchet:circle}, Bouchet points out that his result may be
stated in terms of minors of isotropic systems, which informally are
much closer to $3$-minors of binary delta-matroids than to
vertex-minors of graphs.

\begin{theorem}\label{thm:rg3min}
  A binary delta-matroid is ribbon-graphic if and only if it has no
  $3$-minor that is a twisted dual of $D(G_1)$, $D(G_2)$ or $D(G_3)$.
\end{theorem}

\begin{proof}
  If $D$ is a binary delta-matroid and $v$ is an element of $D$ then
  $D$ is ribbon-graphic if and only if $D+v$ is ribbon graphic,
  because it follows from Theorem 4.1 of~\cite{CCMR} that if $R$ is a
  ribbon graph with $D=D(R)$ then $D+v$ is the delta-matroid
  corresponding to the ribbon graph formed from $R$ by applying a
  half-twist to $v$. Let
  \[ B= \{e\in E(D): \{e\} \in \mathcal {F}(D)\}.\]
  Then $D+B$ is even and, furthermore, $D+B$ is ribbon-graphic if and
  only if $D$ is ribbon-graphic.  Now $D+B=D(G)$ where $G$ is an
  unlooped simple graph. Bouchet's Theorem~\ref{Bou:famouscircle}
  states that $G$ is a circle graph if and only if $G$ has no vertex
  minor isomorphic to $G_1$, $G_2$ or $G_3$. Equivalently $D+B$ is
  ribbon-graphic if and only if it has no $3$-minor that is a twisted
  dual of $D(G_1)$, $D(G_2)$ or $D(G_3)$. As $D+B$ is a twisted dual
  of $D$, the result follows.
\end{proof}

We close by presenting excluded $3$-minor results for the classes of
binary delta-matroids and ribbon graphic delta-matroids that follow
from Theorem~\ref{thm:excl3minvfs}.  Bouchet~\cite{bouchet:binary}
proved that every minor of a binary delta-matroid is binary and gave
the following excluded-minor characterization of binary
delta-matroids.

\begin{theorem}\label{thm:bouchetbinary}
  A delta-matroid is binary if and only if it does not have a minor
  isomorphic to any of the following five delta-matroids or their
  twists.
  \begin{enumerate}
  \item[\emph{(1)}]
    $B_1=(\{a,b,c\},\{\emptyset,\{a,b\},\{a,c\},\{b,c\},\{a,b,c\}\})$;
  \item[\emph{(2)}] $B_2=S_3+\{a,b,c\}$;
  \item[\emph{(3)}]
    $B_3=(\{a,b,c\},\{\emptyset,\{b\},\{c\},\{a,b\},\{a,c\},\{a,b,c\}\})$;
  \item[\emph{(4)}]
    $B_4=(\{a,b,c,d\},\{\emptyset,\{a,b\},\{a,c\},\{a,d\},\{b,c\},\{b,d\},\{c,d\}\})$;
  \item[\emph{(5)}]
    $B_5=(\{a,b,c,d\},\{\emptyset,\{a,b\},\{a,d\},\{b,c\},\{c,d\},\{a,b,c,d\}\})$.
  \end{enumerate}
\end{theorem}

We obtain corollaries of this result characterizing binary
delta-matroids in terms excluded $3$-minors. The first result is
equivalent to a recent result of Brijder, stated in terms of binary
tight $3$-matroids~\cite{Brid:orient}.

\begin{corollary}
  A vf-safe delta-matroid is binary if and only if it has no $3$-minor
  that is a twisted dual of $B_1$.
\end{corollary}

\begin{proof}
  Theorem~8.2 of~\cite{B+H:nullity} states that every twisted dual of
  a binary delta-matroid is a binary delta-matroid. In particular
  every binary delta-matroid is vf-safe.  Moreover, every $3$-minor of
  a binary delta-matroid is binary.  The delta-matroid $B_1$ is
  vf-safe and all of its $3$-minors are binary. Thus all of its
  twisted duals are excluded $3$-minors for the class of binary
  delta-matroids.

  Suppose that $D$ is a vf-safe delta-matroid that is not binary. Then
  Theorem~\ref{thm:bouchetbinary} implies that $D$ has a minor
  isomorphic to a twist of $B_i$ for $1\leq i \leq 5$. The
  delta-matroid $B_2$ is not vf-safe and $B_4 \ddag d=B_2$, so $D$ has
  no minor isomorphic to a twist of $B_2$ or of $B_4$. Furthermore
  $(B_3+a)^*=B_1$, and $B_5\ddag d \simeq B_3$. Thus $D$ has a
  $3$-minor that is isomorphic to a twisted dual of $B_1$.
\end{proof}

By combining this result with Theorem~\ref{thm:excl3minvfs}, we obtain
the following.

\begin{corollary}
  A proper set system is a binary delta-matroid if and only if it has
  no $3$-minor that is a twisted dual of $B_1$ or $S_3$.
\end{corollary}

Finally we combine the last two results with Theorem~\ref{thm:rg3min}.

\begin{corollary}
  A proper set system is a ribbon graphic delta-matroid if and only if
  it has no $3$-minor that is a twisted dual of $B_1$, $S_3$,
  $D(G_1)$, $D(G_2)$ or $D(G_3)$.
\end{corollary}

\section{Appendix:  The twisted duals of $S_3$}

As proved in Theorem~\ref{thm:excl3minvfs}, these twisted duals of
$S_3$ are the excluded $3$-minors for vf-safe delta-matroids.

\begin{table}[h]
\begin{center}
\begin{tabular}{r|cccc|l}
  \cline{2-5}
  $S_3$ & $\emptyset$ &  &  & $\{a,b,c\}$ &\rule{0pt}{10pt}  \\
    \cline{2-5}
   \multicolumn{6}{c}{}\\[-0.7em]
 \cline{2-5}
   $S_3*\{a\}$      & & $\{a\}$ & $\{b,c\}$ & \rule{0pt}{10pt}   \\
 \cline{2-5}
 \multicolumn{6}{c}{}\\[-0.7em]
\end{tabular}
\end{center}
\caption{All twists of $S_3$ up to isomorphism.}\label{t1}
\end{table}

\begin{table}[h]
\begin{center}
\begin{tabular}{|cccc|cccc|}
  \cline{1-8}
 $\emptyset$ & $\{a\}$ &  & $\{a,b,c\}$
  & $\emptyset$ &  & $\{b,c\}$ & $\{a,b,c\}$\rule{0pt}{10pt}  \\
  \cline{1-8}
   \multicolumn{4}{c}{$S_3+\{a\}$}
             &\multicolumn{4}{c}{$(S_3+\{a\})^*$}\rule[-4pt]{0pt}{16pt}
  \\
   \multicolumn{8}{c}{}\\[-0.7em]
  \cline{1-8}
 $\emptyset$ & $\{a\}$ & $\{b,c\}$ &
  &  & $\{a\}$ & $\{b,c\}$ & $\{a,b,c\}$\rule{0pt}{10pt}  \\
  \cline{1-8}
  \multicolumn{4}{c}{  $(S_3+\{a\})*\{a\}$}
             &\multicolumn{4}{c}{$(S_3+\{a\})*\{b,c\}$}\rule[-4pt]{0pt}{16pt}
  \\
 \multicolumn{8}{c}{}\\[-0.7em]
   \cline{1-8}
 & \multirow{2}{*}{$\{b\}$} & $\{a,b\}$ &  &
    & $\{b\}$ & \multirow{2}{*}{$\{a,c\}$} &
  \\
%\hspace{1.6cm}
 &  &  $\{a,c\}$ & & &$\{c\}$ &   &  \\ %\hspace{1.6cm} \\
  \cline{1-8}
  \multicolumn{4}{c}{$(S_3+\{a\})*\{b\}$}  &\multicolumn{4}{c}{$(S_3+\{a\})*\{a,c\}$}\rule[-4pt]{0pt}{16pt} \\
 \multicolumn{8}{c}{}\\[-0.7em]
\end{tabular}
\end{center}
\caption{All twists of $S_3+\{a\}$ up to isomorphism.  Dual pairs
  are side by side.}\label{t2}
\end{table}

\begin{table}[h]
\begin{center}
\begin{tabular}{|cccc|cccc|}
  \cline{1-8}
 \multirow{2}{*}{$\emptyset$} & $\{a\}$ & \multirow{2}{*}{$\{a,b\}$} & \multirow{2}{*}{$\{a,b,c\}$}   &
  \multirow{2}{*}{$\emptyset$} & \multirow{2}{*}{$\{c\}$} & $\{a,c\}$ & \multirow{2}{*}{$\{a,b,c\}$}\rule{0pt}{10pt}\\
& $\{b\}$ &&&&& $\{b,c\}$&\\
    \cline{1-8}
   \multicolumn{4}{c}{$S_3+\{a,b\}$}  &\multicolumn{4}{c}{$(S_3+\{a,b\})^*$}\rule[-4pt]{0pt}{16pt} \\
 \multicolumn{8}{c}{}\\[-0.7em]
  \cline{1-8}
 \multirow{2}{*}{$\emptyset$} & $\{a\}$ & $\{a,b\}$ & &
 & $\{a\}$ & $\{a,c\}$ & \multirow{2}{*}{$\{a,b,c\}$}\rule{0pt}{10pt}  \\
& $\{b\}$ &$\{b,c\}$&&&$\{c\}$& $\{b,c\}$&\\
    \cline{1-8}
      \multicolumn{4}{c}{$(S_3+\{a,b\})*\{a\}$}  &\multicolumn{4}{c}{$(S_3+\{a,b\})*\{b,c\}$}\rule[-4pt]{0pt}{16pt} \\
 \multicolumn{8}{c}{}\\[-0.7em]

  \cline{1-8}
&&$\{a,b\}$&&&$\{a\}$&&\\
 & $\{c\}$ & $\{a,c\}$ & $\{a,b,c\}$  &
  $\emptyset$ & $\{b\}$ & $\{a,b\}$ &   \\
&&$\{b,c\}$&&&$\{c\}$&&\\
    \cline{1-8}
      \multicolumn{4}{c}{$(S_3+\{a,b\})*\{c\}$}  &\multicolumn{4}{c}{$(S_3+\{a,b\})*\{a,b\}$}\rule[-4pt]{0pt}{16pt} \\
 \multicolumn{8}{c}{}\\[-0.7em]
\end{tabular}
\end{center}
\caption{All twists of $S_3+\{a,b\}$ up to isomorphism.  Dual pairs
  are side by side.}\label{t3}
\end{table}

\begin{table}[h]
\begin{center}
\begin{tabular}{|cccc|cccc|}
\cline{1-8}
&$\{a\}$&$\{a,b\}$&&&$\{a\}$&$\{a,b\}$&\\
 $\emptyset$ & $\{b\}$ & $\{a,c\}$ &  &
   & $\{b\}$ & $\{a,c\}$ & $\{a,b,c\}$   \\
&$\{c\}$&$\{b,c\}$&&&$\{c\}$&$\{b,c\}$&\\
    \cline{1-8}
      \multicolumn{4}{c}{  $S_3+\{a,b,c\}$}  &\multicolumn{4}{c}{$(S_3+\{a,b,c\})^*$}\rule[-4pt]{0pt}{16pt} \\
 \multicolumn{8}{c}{}\\[-0.7em]
  \cline{1-8}
&$\{a\}$&\multirow{2}{*}{$\{a,b\}$}&&&\multirow{2}{*}{$\{b\}$}&$\{a,b\}$&\\
  $\emptyset$&$\{b\}$&\multirow{2}{*}{$\{a,c\}$}&$\{a,b,c\}$&
  $\emptyset$&\multirow{2}{*}{$\{c\}$}&$\{a,c\}$&$\{a,b,c\}$  \\
&$\{c\}$&&&&&$\{b,c\}$&\\
    \cline{1-8}
      \multicolumn{4}{c}{$S_3+\{a,b,c\}*\{a\}$ }  &\multicolumn{4}{c}{$S_3+\{a,b,c\}*\{b,c\}$}\rule[-4pt]{0pt}{16pt}\\
 \multicolumn{8}{c}{}\\[-0.7em]
\end{tabular}
\end{center}
\caption{All twists of $S_3+\{a,b,c\}$ up to isomorphism.  Dual pairs
  are side by side.}\label{t4}
\end{table}

\begin{table}[!h]
\begin{center}
\begin{tabular}{|cccc|cccc|}
  \cline{1-8}
           &\multirow{2}{*}{$\{a\}$}&$\{a,b\}$&\multirow{2}{*}{$\{a,b,c\}$}&
  \multirow{2}{*}{$\emptyset$}&$\{a\}$&\multirow{2}{*}{$\{b,c\}$}&  \\
&&$\{b,c\}$&&&$\{c\}$&&\\
  \cline{1-8}
      \multicolumn{4}{c}{ $(S_3*\{a\})+\{a,b\}$  }  &\multicolumn{4}{c}{$((S_3*\{a\})+\{a,b\})^*$}\rule[-4pt]{0pt}{16pt}\\
 \multicolumn{8}{c}{}\\[-0.7em]
\end{tabular}\\
\begin{tabular}{|cccc|}
  \hline
  $\emptyset$&$\{b\}$&$\{b,c\}$&$\{a,b,c\}$\\
    \hline
 \multicolumn{4}{c}{$((S_3*\{a\})+\{a,b\})*\{a\}$}\rule[-4pt]{0pt}{16pt}\\
 \multicolumn{4}{c}{}\\[-0.7em]
  \hline
  &$\{a\}$&$\{a,b\}$&\\&$\{c\}$&$\{a,c\}$&\\
    \hline
 \multicolumn{4}{c}{$((S_3*\{a\})+\{a,b\})*\{b\}$}\rule[-4pt]{0pt}{16pt}\\
 \multicolumn{4}{c}{}\\[-0.7em]
\end{tabular}
\end{center}
\caption{All twists of $(S_3*\{a\})+\{a,b\}$ up to isomorphism.  Dual pairs
  are side by side.}\label{t5}
\end{table}

\begin{table}[h]
\begin{center}
\begin{tabular}{|cccc|cccc|}
  \cline{1-8}
  &&$\{a,b\}$&&&$\{a\}$&&\\
  &$\{a\}$&$\{a,c\}$&&&$\{b\}$&$\{b,c\}$&  \\
&&$\{b,c\}$&&&$\{c\}$&&\\
    \cline{1-8}
      \multicolumn{4}{c}{   $(S_3*\{a\})+\{a,b,c\}$ }  &\multicolumn{4}{c}{$((S_3*\{a\})+\{a,b,c\})^*$}\rule[-4pt]{0pt}{16pt}\\
 \multicolumn{8}{c}{}\\[-0.7em]
  \cline{1-8}
        \multirow{2}{*}{$\emptyset$}&$\{b\}$&&\multirow{2}{*}{$\{a,b,c\}$}&
\multirow{2}{*}{$\emptyset$}&&$\{a,b\}$&\multirow{2}{*}{$\{a,b,c\}$}\\
&$\{c\}$&&&&&$\{a,c\}$&\\
    \cline{1-8}
      \multicolumn{4}{c}{$((S_3*\{a\})+\{a,b,c\})*\{a\}$}  &\multicolumn{4}{c}{$((S_3*\{a\})+\{a,b,c\})*\{b,c\}$}\rule[-4pt]{0pt}{16pt}\\
 \multicolumn{8}{c}{}\\[-0.7em]
  \cline{1-8}

&$\{a\}$&\multirow{2}{*}{$\{a,b\}$}&\multirow{2}{*}{$\{a,b,c\}$}&
\multirow{2}{*}{$\emptyset$}&\multirow{2}{*}{$\{c\}$}&$\{a,b\}$&\\

&$\{c\}$&&&&&$\{b,c\}$&\\
    \cline{1-8}
      \multicolumn{4}{c}{$((S_3*\{a\})+\{a,b,c\})*\{b\}$}  &\multicolumn{4}{c}{$((S_3*\{a\})+\{a,b,c\})*\{a,c\}$}\rule[-4pt]{0pt}{16pt}\\
 \multicolumn{8}{c}{}\\[-0.7em]
\end{tabular}
\end{center}
\caption{All twists of $(S_3*\{a\})+\{a,b,c\}$ up to isomorphism.  Dual pairs
  are side by side.}\label{t6}
\end{table}

\end{document}